\documentclass[reqno,11pt]{article}

\usepackage{amsfonts,amssymb,latexsym,epsf,amsmath,amsthm}

\newtheorem{thm}{Theorem}[section]
\newtheorem{cor}[thm]{Corollary}
\newtheorem{lem}[thm]{Lemma}
\newtheorem{prop}[thm]{Proposition}
\newtheorem{defn}[thm]{Definition}
\newtheorem{rem}[thm]{Remark}

\numberwithin{equation}{section}

\def\Om{\Omega}
\def\R{\mathbb{R}}

\def\Lqx{L^{q(\cdot)}(\Om)}
\def\Wpx{W^{1,p(\cdot)}_0(\Om)}
\def\Wp{W^{1,p}_0(\Om)}
\def\Wqx{W^{1,q(\cdot)}_0(\Om)}

\def\W11{W^{1,1}_0(\Om)}

\def\convex{{\mathcal K}_\psi \cap L^{\infty} (\Om)}
\def\convexn{{\mathcal K}_{\psi_n} \cap L^{\infty} (\Om)}

\def\entropycond{{\rm(}1.5{\rm})_{f,\psi}}
\def\entropycondn{{\rm(}1.5{\rm})_{f_n,\psi_n}}
\def\entropycondfn{{\rm(}1.5{\rm})_{f_n,\psi}}
\def\entropycondfin{{\rm(}1.5{\rm})_{f_n^i,\psi}}

\def\build#1_#2^#3{\mathrel{\mathop{\kern 0pt#1}\limits_{#2}^{#3}}}

\begin{document}

\title{The obstacle problem for nonlinear elliptic equations with variable growth
and $L^1-$data\footnote{accepted for publication in
\textsf{Monatsh.~Math.} The original article will be available at
\texttt{www.springerlink.com}}}

\author{
{\sc Jos\'e Francisco Rodrigues}\\
{\sc\footnotesize CMUC, Department of Mathematics, University of
Coimbra}\\
{\sc\footnotesize and FCUL/Universidade de Lisboa}\\
{\sc \footnotesize Av. Prof. Gama Pinto 2, 1649--003 Lisboa, Portugal}\\
{\tt\footnotesize rodrigue@ptmat.fc.ul.pt}
\and
\sc Manel Sanch\'on\footnote{Current address: Universitat de
Barcelona, Departament de Matem\`atica Aplicada i An\`alisi, Gran
Via 585, 08007 Barcelona, Spain. \textit{E-mail}: {\tt msanchon@maia.ub.es}}\\
{\sc\footnotesize CMUC, Department of Mathematics, University of Coimbra}\\
{\sc \footnotesize 3001--454 Coimbra, Portugal}
\and
\sc Jos\'e Miguel Urbano\\
{\sc\footnotesize CMUC, Department of Mathematics, University of Coimbra}\\
{\sc \footnotesize 3001--454 Coimbra, Portugal}\\
{\tt\footnotesize jmurb@mat.uc.pt}
}

\date{}

\maketitle

\begin{abstract}
The aim of this paper is twofold: to prove, for $L^1-$data, the
existence and uniqueness of an entropy solution to the obstacle
problem for nonlinear elliptic equations with  variable growth, and
to show some convergence and stability properties of the
corresponding coincidence set. The latter follow from extending the
Lewy--Stampacchia inequalities to the general framework of $L^1$.

\bigskip



\end{abstract}

\section{Introduction}\label{section1}

Let $\Om \subset \R^N$, $N \geq 1$, be a bounded domain. The purpose
of this paper is the study of the obstacle problem associated with
nonlinear elliptic equations with data $f \in L^1(\Om)$ and
principal part modeled on the $p(\cdot)-$Laplacian with variable
exponent
$$\Delta_{p(x)} u := {\rm div}\  |\nabla u|^{p(x)-2} \nabla u .$$
These obstacle problems fall into the framework of the model problem
\begin{equation}
\left\{
\begin{array}{rclll}
-\Delta_{p(\cdot)} u + \beta (\cdot,u) &= & f & \ \textrm{in} & \Om, \\
u & = & 0  & \ \textrm{on} & \partial \Om,
\end{array}\right.
\label{mp}
\end{equation}
for a certain function $\beta$, related to a maximal monotone graph.
For instance, in the case of the zero obstacle problem, when $u\geq
0\, $ a.e. in $\Om$, it can be shown that $\beta$ is a.e. given by
the nonlinear discontinuity
\begin{equation}
\beta(x,u)= \left\{
\begin{array}{ccl}
0 & \ \textrm{if} & \ u(x)>0,\\
-f^-(x) & \ \textrm{if} & \ u(x)=0,
\end{array} \right.
\label{beta}
\end{equation}
where $f^-$ is the negative part of the decomposition $f=f^+-f^-$.
Problems of the type \eqref{mp} have been solved by Br\'ezis and
Strauss \cite{BS} for linear elliptic operators ($p(\cdot) \equiv
2$) and general maximal monotone graphs $\beta$. An $L^1-$theory for
the case of $p-$Laplacian type operators (with $p$ constant) has
been proposed in \cite{BBGGPV95} by B\'enilan \textit{et als.},
\textit{via} the introduction of the notion of entropy solution.
This framework has been extended to unilateral problems with
constant $p$ in \cite{BG}, \cite{BC}, and \cite{P}. The interesting
cases are those of $1<p\leq N$, since for $p>N$ the variational
methods of Leray--Lions (see, for instance, \cite{L69}) easily
apply, the solution being bounded and with gradient in $L^p(\Om)$.
Recently, the obstacle problem with more general data, namely with
$f$ only a measure, has been considered by several authors (see,
\textit{e.g.}, \cite{DADM,L1,L2,BP}). In particular, Br\'ezis and
Ponce show in \cite{BP}, still in the case $p=2$ and for a constant
obstacle, that $f^- \in L^1(\Om) + H^{-1} (\Om)$ is a necessary and
sufficient condition for the existence (and the uniqueness) of a
solution to \eqref{mp}.

On the other hand, for the case of a variable exponent, the
existence and uniqueness of an entropy solution to \eqref{mp}, with
$\beta \equiv 0$ and $f \in L^1$, has been recently obtained by two
of the authors in \cite{SU06}. The result builds upon
\cite{BBGGPV95} and \cite{ABFOT03}, assumes the exponent to be
log-H\"{o}lder
continuous, and relies on \textit{a priori} estimates in
Marcinkiewicz spaces with variable exponent. A primary aim of this
paper is to extend this theory to obstacle problems ($u\geq \psi$ in
$\Om$), for admissible general obstacles $\psi=\psi(x)$ and
nonlinear operators with variable growth.

The natural framework to solve problem \eqref{mp} is that of Sobolev
spaces with variable exponent. Recent applications in elasticity
\cite{Zh}, non--Newto\-nian fluid mechanics \cite{Zh2,Ruzicka00,AR},
or image processing \cite{CLR}, gave rise to a revival of the
interest in these spaces, the origins of which can be traced back to
the work of Orlicz in the 1930's. An account of recent advances,
some open problems, and an extensive list of references can be found
in the interesting surveys by Diening \textit{et als.} \cite{DHN}
and Antontsev \textit{et al.} \cite{AS} (\textit{cf.} also the work
of Kov\'{a}\v{c}ik and R\'{a}kosn\'{\i}k \cite{KR}, where many of
the basic properties of these spaces are established). A brief
introduction to the subject, which is pertinent to the present paper
can be found in \cite{SU06}; we will refer the reader to this paper,
when appropriate, to avoid an unnecessary duplication of arguments.

For quasilinear operators in divergence form of $p(\cdot)-$Laplacian
type
$${\mathcal A}u := - {\rm div}\ a(x,\nabla u),$$
the classical obstacle problem can be formulated, using the duality
between $W_0^{1,p(\cdot)} (\Om)$ and $W^{-1,p^\prime(\cdot)} (\Om)$,
in terms of the variational inequality
\begin{equation}
u \in {\cal K}_\psi \ : \ \int_\Om a(x,\nabla u) \cdot \nabla (v-u)
\, dx \geq \langle f,v-u \rangle \ , \quad \forall v \in {\cal
K}_\psi, \label{vi}
\end{equation}
whenever $f \in W^{-1,p^\prime(\cdot)} (\Om)$ and the convex subset
\begin{equation}
{\cal K}_\psi = \left\{ v \in W_0^{1,p(\cdot)} (\Om) \ : \ v \geq
\psi \ \: \textrm{a.e. in} \ \Om\right\} \label{compact}
\end{equation}
is nonempty. The former holds in the case $f \in L^1(\Om)$ and
$p(\cdot)>N$ (since then, by Sobolev's embedding, $W_0^{1,p(\cdot)}
(\Om) \subset L^\infty (\Om)$) or if $f \in L^{r(\cdot)} (\Om)$,
with $N/p(\cdot) < r(\cdot)$, for $1< p(\cdot) < N$. The theory of
monotone operators then applies to \eqref{vi} (see
\cite{L69,KinStam80}), with
$$\langle f,v-u \rangle = \int_\Om f(v-u) \: dx.$$
As in the case of a constant $p$, for $f \in L^1(\Om)$ and
$1<p(\cdot)<N$, both sides of inequality \eqref{vi} may have no
meaning, so we are led, following \cite{BC} (\textit{cf.} also
\cite{BG} and \cite{P}), to extend the formulation of the unilateral
problem by replacing $v-u$ by its truncation $T_t(u-v)$, for every
level $t>0$, where $T_t$ is defined by
$$T_t(s):=\max \left\{ -t,\min\{t,s\} \right\},\quad s\in \R.$$
The resulting notion of entropy solution for the obstacle problem is
made precise in the following definition.

\begin{defn} An entropy solution
of the obstacle problem for $\{f,\psi\}$ is a measurable function
$u$ such that $u\geq \psi$ a.e. in $\Omega$, and, for every $t>0$,
$T_t(u)\in \Wpx$ and
$$\int_\Om a(x,\nabla u)\cdot \nabla T_t(\varphi-u)\ dx\geq \int_\Om
f \: T_t(\varphi-u)\ dx , \eqno{{\rm(}1.5{\rm})_{f,\psi}}$$
\stepcounter{equation} for all $\varphi \in {\cal K}_\psi \cap
L^\infty(\Om) $. \label{defop}
\end{defn}

This entropic formulation is adequate since we are able to show the
existence and uniqueness of a solution. In general, entropic
solutions do not belong to ${\cal K}_\psi$, since they do not have
an integrable distributional gradient; if $1<p(\cdot) < 2-1/N$, they
may not even be $L^1-$functions. However, they belong to
$W_0^{1,1}(\Om)$ if $p(\cdot) > 2-1/N$.

The framework is also adequate in order to obtain the continuous
dependence of the solution with respect to variations of the
obstacle in $W^{1,p(\cdot)}(\Om)$ and of the nonhomogeneous term in
$L^1 (\Om)$, extending the results of \cite{Cirmi} concerning the
constant exponent case.

For constant $p$, and certain assumptions on $f$ and ${\mathcal
A}\psi$, implying that ${\mathcal A}u \in L^1(\Om)$, it has been
observed in \cite{Ro05} that the variational solution to \eqref{vi}
actually satisfies, a.e. in $\Om$, an equation with a nonlinear
discontinuity and, in particular, that
\begin{equation}
{\mathcal A}u = f \ , \quad \textrm{a.e. in} \ \ \{ u > \psi\},
\label{Au=f}
\end{equation}
where $\{ u > \psi\} = \Om \setminus \{ u = \psi\}$ is the
complement of the coincidence set $\{ u = \psi\} := \{ x \in \Om \ :
\ u(x) = \psi(x)\}$. In fact, in the free boundary domain
$\{u>\psi\}$, equation \eqref{Au=f} can be obtained as a consequence
of the well-known Lewy--Stampacchia inequalities
\begin{equation}\label{Lewy-Stam}
f\leq {\mathcal A}u\leq f+({\mathcal A}\psi-f)^+ \ , \quad
\textrm{a.e. in }\Omega.
\end{equation}
A second main result we obtain in this paper is the extension of
these assertions to the general framework of entropy solutions of
equations involving variable exponents. In particular, for the
obstacle problem with an admissible obstacle $\psi$ such that
$({\mathcal A}\psi-f)^+\in L^1(\Omega)$, we show, still in the
$L^1-$framework, that in \eqref{mp},
$$\beta(\cdot,u)=-({\mathcal A}
\psi-f)^+\chi_{\{u=\psi\}} \ , \quad \textrm{a.e. in }\Omega,$$
where $\chi_{S}$ denotes the characteristic function of the set $S$.
In the special case $\psi\equiv 0$, we obtain \eqref{beta}.

An important consequence of inequalities \eqref{Lewy-Stam} is the
reduction of the regularity issue for the solutions of the obstacle
problem to that of the solutions of the corresponding equations. In
particular, we conclude that the boundedness of $f$ and $({\mathcal
A}\psi-f)^+$ are sufficient to guarantee the local H\"older
continuity of the solution and its gradient for the
$p(\cdot)-$obstacle problem, in accordance with the case of
equations (see \cite{Al} and \cite{FZ}) or that of functionals with
non--standard growth conditions (\cite{AM01}).

We also extend, for a fixed admissible obstacle $\psi$, the
$L^1-$contraction property of Br\'ezis and Strauss \cite{BS} for the
map $f\longmapsto\beta_f$. The property was obtained by one of the
authors for quasilinear obstacle problems (see \cite{Ro87,Ro05}),
with the aim of estimating the stability of two coincidence sets
$\{u_1=\psi\}$ and $\{u_2=\psi\}$ with respect to the $L^1-$norm of
the difference $f_1-f_2$ of the corresponding variational data. The
extension of these results to entropy solutions, in the context of
data merely in $L^1$, places the stability theory of the coincidence
sets (with respect to the variation of non-degenerate data) in its
natural and more general framework.

The paper is organized as follows. In section $2$, we introduce the
assumptions and state the main results. In section \ref{section2},
we prove \textit{a priori} estimates for an entropy solution of the
obstacle problem. Section~\ref{section3} deals with the existence
and uniqueness of an entropy solution and its continuous dependence
with respect to the data. In section \ref{section4}, we extend
Lewy--Stampacchia inequalities to the context of entropy solutions
and analyze their consequences, namely the characterization of the
obstacle problem in $L^1$ in terms of an equation with a nonlinear
discontinuity, and the stability of the coincidence sets.


\section{Main results}\label{section1.1}

Let $a:\Om\times\R^N\rightarrow\R^N$ be a Carath\'eodory function
(\textit{i.e.}, $a(\cdot,\xi)$ is measurable on $\Om$, for every
$\xi\in\R^N$, and $a(x,\cdot)$ is continuous on $\R^N$, a.e.
$x\in\Om$), such that the following assumptions hold:
\begin{equation}\label{assumption1}
a(x,\xi)\cdot\xi\geq \alpha|\xi|^{p(x)},
\end{equation}
a.e. $x\in\Om$, for every $\xi\in\R^N$, where $\alpha$ is a positive
constant;
\begin{equation}\label{assumption2}
|a(x,\xi)|\leq \gamma \left( j(x)+|\xi|^{p(x)-1} \right),
\end{equation}
a.e. $x\in\Om$, for every $\xi\in\R^N$, where $j$ is a nonnegative
function in $L^{p'(\cdot)}(\Om)$ and $\gamma>0$;
\begin{equation}\label{assumption3}
(a(x,\xi)-a(x,\xi'))\cdot(\xi-\xi')>0,
\end{equation}
a.e. $x\in\Om$, for every $\xi,\xi'\in \R^N$, with $\xi\neq\xi'$.

These are standard assumptions when dealing with monotone operators
in divergence form, the novelty being the fact that the exponent
$p(\cdot)$, appearing in \eqref{assumption1} and
\eqref{assumption2}, does not need to be constant but may depend on
the variable $x$. Throughout the paper, the following notation for a
measurable function $q(\cdot):\Omega\rightarrow\mathbb{R}$ will be
used:
$$\underline{q}:= \ \build {\textstyle \rm \mbox{ess
inf}}_{\scriptstyle x\in\Om}^{} \: q(x) \qquad \textrm{and} \qquad
\overline{q}:= \ \build {\textstyle \rm \mbox{ess
sup}}_{\scriptstyle x\in\Om}^{} \: q(x).$$ The exponent is assumed
here to be a measurable function $p(\cdot):\Om\rightarrow \R$ such
that
\begin{equation}\label{assumptionp(x)}
\left\{ \begin{array}{l}
\exists C>0 \ : \  |p(x)-p(y)|\leq\frac{C}{-\ln|x-y|},  \quad \textrm{for} \ |x-y|<\frac{1}{2};\\
\\
1<\underline{p} \leq \overline{p} <N.
\end{array} \right.
\end{equation}
The first  condition says that $p(\cdot)$ is a log-H\"older
continuous function. On the other hand,
the second assumption in \eqref{assumptionp(x)} is quite natural if
one wants to define an appropriate functional setting.
Assumption \eqref{assumptionp(x)}
puts us in the framework of reflexive Sobolev spaces with variable
exponent and allows us to exploit their properties, like the crucial
Poincar\'e and Sobolev inequalities. These generalized
Sobolev-Orlicz spaces consist of measurable functions $v:\Om
\rightarrow \R$, such that $v$ and its distributional gradient
$\nabla v$ are in $L^{p(\cdot)} (\Om)$, the space of functions with
finite modular
$$\varrho_{p(\cdot)}(v) = \int_\Om |v(x)|^{p(x)}\, dx,$$
normed by
$$\|v\|_{p(\cdot)} = \inf \left\{ \lambda >0 \, : \, \varrho_{p(\cdot)}
(v/\lambda) \leq 1\right\}.$$ Under assumption
\eqref{assumptionp(x)}, the variable exponent Lebesgue spaces have
properties similar to those of the classical Lebesgue spaces, being
reflexive and separable Banach spaces, and satisfying the continuous
embedding $L^{p(\cdot)} (\Om) \hookrightarrow L^{q(\cdot)} (\Om)$,
for $\Om$ bounded and $p(x)\geq q(x)$. These spaces are not
invariant to translations (see \cite{KR}) although a H\"{o}lder type
inequality holds. For Sobolev spaces with variable exponent, we can
define $W^{-1,p^\prime(\cdot)} (\Om)$ as the dual space of
$W_0^{1,p(\cdot)} (\Om)$, where Poincar\'e's inequality is also
valid. Besides, the Sobolev embedding
$$W_0^{1,p(\cdot)} (\Om) \hookrightarrow L^{p^\ast(\cdot)} (\Om) \ ,
\qquad p^\ast(\cdot)=\frac{Np(\cdot)}{N-p(\cdot)}$$ still holds (see
\cite{ER00,D,HHKV}).
Let us finally introduce the following notation: given two bounded
measurable functions $p(\cdot),q(\cdot) : \Om \rightarrow \R$, we
write
$$q(\cdot) \ll p(\cdot) \qquad \textrm{if} \qquad \build {\textstyle \rm
\mbox{ess inf}}_{\scriptstyle x\in\Om}^{} \: \left( p(x)-q(x)
\right) > 0.$$

Concerning the right-hand side of $\entropycond$ and the obstacle
$\psi$ we make the following assumptions:
\begin{equation}
f\in L^1(\Om),\quad \psi\in W^{1,p(\cdot)}(\Om), \quad \textrm{and}
\quad \psi^+ \in\Wpx \cap L^\infty (\Om). \label{assumptionf}
\end{equation}
In particular, the last assumption guarantees that $\convex\neq
\emptyset$.

Our first result concerns the existence and uniqueness of an entropy
solution, in the sense of Definition \ref{defop}, to the obstacle
problem; we also obtain regularity results for the solution and its
weak gradient. We recall from \cite{SU06} that it is still possible,
as in the case of a constant $p$ (cf. \cite{BBGGPV95}), to define
the weak gradient of a measurable function $u$ such that $T_t(u)\in
\Wpx$, for all $t>0$. In fact, there exists a unique measurable
vector field $\mathbf{v}:\Om\rightarrow \R^N$ such that
$$\mathbf{v}\chi_{\{|u|<t\}}=\nabla T_t(u), \quad \textrm{a.e. in } \Om, \quad \textrm{for all } t>0.$$
Moreover, if $u \in W^{1,1}_0(\Om)$ then $\mathbf{v}$ coincides with
$\nabla u$, the standard distributional gradient of $u$.

Assuming \eqref{assumptionp(x)}, the critical Sobolev exponent and
the conjugate of $p(\cdot)$ are, respectively,
$$
p^*(\cdot)=\frac{Np(\cdot)}{N-p(\cdot)}\quad\textrm{and}\quad
p'(\cdot)=\frac{p(\cdot)}{p(\cdot)-1}.
$$

The following is one of our main results.

\begin{thm}\label{thm1}
Assume \eqref{assumption1}--\eqref{assumptionf}. There exists a
unique entropy solution $u$ to the obstacle problem $\entropycond$.
Moreover, $|u|^{q(\cdot)}\in L^1(\Om)$, for all $0\ll q(\cdot)\ll
q_0(\cdot)$, and $|\nabla u|^{q(\cdot)}\in L^1(\Om)$, for all $0\ll
q(\cdot) \ll q_1(\cdot)$, where
\begin{equation}\label{q0andq1}
q_0(\cdot):=\frac{p^*(\cdot)}{\overline{p'}}\quad\textrm{and}\quad
q_1(\cdot):=\frac{q_0(\cdot)}{q_0(\cdot)+1}p(\cdot).
\end{equation}
In particular, if $\ 2-1/N \ll p(\cdot)$ then
$$u\in \Wqx, \quad \textrm{for all} \ \ 1\leq q(\cdot)\ll q_1(\cdot).$$
\end{thm}

\begin{rem}
{\rm Among other results, Boccardo and Cirmi prove in \cite{BC} an
analogous of Theorem \ref{thm1}, for constant $p(\cdot)\equiv p >
2-1/N$, and under the assumption that $\psi\in \Wp \cap
L^\infty(\Om)$. Under our assumptions, since $\psi^+ \in\Wpx \cap
L^\infty (\Om)$, $\psi$ is bounded above but not necessarily bounded
below.} \label{rem:thm1}
\end{rem}

\begin{rem}
{\rm Similar results of existence of entropy solutions for
$L^1-$data could be obtained for more general elliptic operators
with variable growth in the form
$$
\mathcal{A}u=-{\rm div}\, a(x,u,\nabla u)+H(x,u,\nabla u),
$$
where $H$ has the natural growth with respect to the gradient; these
would follow as an extension of the recent results obtained in
\cite{AA06} for constant $p$.
} \label{rem2:thm1}
\end{rem}

We now consider a sequence $\{f_n,\psi_n\}_n$ and the corresponding
obstacle problems $\entropycondn$. The next result states that,
under adequate assumptions, the limit of entropy solutions $u_n$ of
$\entropycondn$ is the solution of the limit obstacle problem
$\entropycond$.

\begin{thm}\label{thm2}
Let $\{f_n,\psi_n\}_n$ be a sequence in $L^1(\Om) \times
W^{1,p(\cdot)}(\Om)$. Assume
\eqref{assumption1}--\eqref{assumptionf} and that ${\psi_n}^+
\in\Wpx \cap L^\infty (\Om)$, for all $n$. Let $u_n$ be the entropy
solution of the obstacle problem ${\rm(}1.5{\rm})_{f_n,\psi_n}$. If
\begin{equation}\label{convergence:fn}
f_n  \longrightarrow f \quad \textrm{in }L^1 (\Om) \qquad and \qquad
\psi_n \longrightarrow \psi \quad \textrm{in }W^{1,p(\cdot)}(\Om),
\end{equation}
then
$$u_n  \longrightarrow u \quad \textrm{in measure},$$
where $u$ is the unique entropy solution of the obstacle problem
$\entropycond$. If $2-1/N \ll p(\cdot)$ then
$$u_n  \rightharpoonup u \quad \textrm{in}\quad \Wqx, \quad \textrm{for all}
\ 1\leq q(\cdot)\ll q_1(\cdot).$$
\end{thm}

We also establish the so-called Lewy--Stampacchia inequalities and
deduce from them a few interesting properties. These results require
the extra assumption $p(\cdot)-1\ll q_1(\cdot)$, which is necessary
in order to pass to the limit in a sequence of approximated obstacle
problems (see Proposition \ref{L1convergence} and Remark
\ref{new_rem}).

\begin{thm}\label{thm3}
Assume \eqref{assumption1}--\eqref{assumptionf},
${\mathcal A}\psi\in L^1 (\Om)$, and $p(\cdot)-1\ll q_1(\cdot)$.
Let $u$ be the entropy solution of the obstacle problem
$\entropycond$. Then ${\mathcal A}u \in L^1(\Om)$ and the following
Lewy--Stampacchia inequalities hold
\begin{equation}
f \leq {\mathcal A}u \leq f+({\mathcal A}\psi-f)^+, \quad \mbox{a.e.
in } \Om. \label{LeviStam_ineq}
\end{equation}
\end{thm}
The most immediate consequences of the Lewy--Stampacchia
inequalities concern the regularity of solutions. If $f$, ${\mathcal
A}\psi\in L^{m(\cdot)}(\Om)$, with
$m(\cdot)=(p^\ast(\cdot))^\prime$, then the entropy solution $u$ of
the obstacle problem $\entropycond$ is the variational solution $u
\in W^{1,p(\cdot)}_0 (\Om)$ of \eqref{vi}, for which the following
regularity assertions hold.

\begin{prop}
Assume \eqref{assumption1}--\eqref{assumptionf}
and $p(\cdot)-1\ll q_1(\cdot)$.
If $f$, ${\mathcal A}\psi \in L^\infty (\Om)$ then the solution
$u$ of \eqref{vi} is such that $u \in L^\infty (\Om) \cap
C^{0,\alpha} (\Om)$. If, in addition, $\partial \Om \in C^{0,1}$
then $u \in C^{0,\alpha} (\overline{\Om})$.

Moreover, in the case that ${\mathcal A} \equiv \Delta_{p(\cdot)}$,
we further have $u \in C_{\mathrm{loc}}^{1,\alpha'} (\Om)$.
\end{prop}
The first part is a straightforward consequence of \cite[Theorems
4.2--4.4]{FZ}, where the H\"{o}lder continuity of weak solutions of
quasilinear elliptic equations with variable growth is obtained; the
second part follows from \cite[Theorem 2.2]{AM01}, that concerns the
H\"{o}lder continuity of the gradient of local minimizers of
functionals with non-standard growth.

Using the Lewy--Stampacchia inequalities and showing that ${\mathcal
A}u=f$, a.e. in $\{u>\psi\}$, we prove that the entropy solution of
$\entropycond$ satisfies an equation involving the coincidence set
$\{u=\psi\}$.

\begin{thm}\label{thm4}
Assume \eqref{assumption1}--\eqref{assumptionf},
${\mathcal A}\psi\in L^1(\Om)$, and $p(\cdot)-1\ll q_1(\cdot)$.
The entropy solution $u$ of the obstacle problem $\entropycond$
satisfies the equation
\begin{equation}
{\mathcal A}u - ({\mathcal A}\psi-f) \chi_{\{u=\psi\}} =f \ ,\quad
\mbox{a.e. in } \Om. \label{eq}
\end{equation}
\end{thm}
We note that \eqref{LeviStam_ineq} and \eqref{eq} imply, in
particular,
$$({\mathcal A}\psi-f) \chi_{\{u=\psi\}} = ({\mathcal A}
\psi-f)^+\chi_{\{u=\psi\}} \ , \quad \textrm{a.e. in }\Om.$$

The next result establishes the convergence of the coincidence set
of a sequence of entropy solutions to the limit coincidence set.

\begin{thm}\label{thm5}
Under the assumptions of Theorem {\rm \ref{thm2}}, assume that
$p(\cdot)-1\ll q_1(\cdot)$,
$${\mathcal A}\psi_n \longrightarrow {\mathcal A}\psi \quad \textrm{in }L^1(\Om) \qquad and
\qquad {\mathcal A}\psi\neq f, \quad \textrm{a.e. in }\Om.$$ Then
\begin{equation}\label{limit_coincidence}
\chi_{\{u_n=\psi_n\}}  \longrightarrow  \chi_{\{u=\psi\}} \quad
\textrm{in }L^{q}(\Om),
\end{equation}
for all $1\leq q<+\infty$.
\end{thm}

Finally, we obtain an $L^1$-contraction property for the obstacle
problem and an estimate for the stability of two coincidence sets
$I_1$ and $I_2$ in terms of their symmetric difference
$$I_1\div I_2:=(I_1\setminus I_2)\cup (I_2\setminus I_1).$$
The results were known for more regular solutions (\textit{cf.}
\cite{Ro87,Ro05}) but meet their natural and more general
formulation in the context of entropy solutions for data precisely
in $L^1(\Om)$.

\begin{thm}\label{thm6}
Assume \eqref{assumption1}--\eqref{assumptionp(x)}
and $p(\cdot)-1\ll q_1(\cdot)$.
Let $f_1$, $f_2\in L^1(\Om)$ and let $\psi$ satisfy
\eqref{assumptionf} and ${\mathcal A}\psi\in L^1(\Om)$. Let $u_1$
and $u_2$ be the entropy solutions of the obstacle problems
${\rm(}1.5{\rm})_{f_1,\psi}$ and ${\rm(}1.5{\rm})_{f_2,\psi}$,
respectively. If $\xi_i:=f_i-{\mathcal A}u_i$, $i=1,2$, then
\begin{equation}\label{L1}
\|\xi_1-\xi_2\|_1\leq \|f_1-f_2\|_1.
\end{equation}
If, in addition, the non-degeneracy condition
\begin{equation}\label{assumptionX}
f_i-{\mathcal A}\psi\leq-\lambda<0, \quad \textrm{a.e. in }D,\
i=1,2,
\end{equation}
holds in a measurable subset $D \subset \Om$, then, for
$I_i:=\{u_i=\psi\}$,
\begin{equation}
{\rm meas}\left((I_1\div I_2)\cap D\right)\leq \frac{1}{\lambda}
\|f_1-f_2\|_1.
\end{equation}
\end{thm}

\section{A priori estimates}\label{section2}

The main purpose of this section is to obtain \textit{a priori}
estimates in Marcin\-kiewicz spaces with variable exponent for an
entropy solution of the obstacle problem $\entropycond$. In face of
the embedding results of \cite{SU06}, we then derive \textit{a
priori} estimates in Lebesgue spaces with variable exponent. We
recall the definition of Macinkiewicz spaces with variable exponent
introduced in \cite{SU06}.

\begin{defn}
Let $q(\cdot)$ be a measurable function such that $\underline{q}>0$.
We say that a measurable function $u$ belongs to the
\textit{Marcinkiewicz space} $M^{q(\cdot)}(\Om)$ if there exists a
positive constant $M$ such that
$$\int_{\{|u|> t\}} t^{q(x)} \ dx\leq M,\quad \textrm{for all}\ \ t>0.$$
\label{defn:Marcinkiewicz}
\end{defn}

The following result is instrumental in obtaining \textit{a priori}
estimates for the obstacle problem.

\begin{lem}\label{lem3:2}
Assume \eqref{assumption1}--\eqref{assumptionf} and let
$\varphi\in\convex$. If $u$ is an entropy solution of the
variational inequality $\entropycond$ then
$$\int_{\{|u|\leq t\}}\!\!\!\!|\nabla u|^{p(x)} dx \leq
C\!\!\left(\!\!\left( t+\|\varphi\|_\infty \right)\|f\|_1+
\!\!\int_\Om\left(|\nabla\varphi|^{p(x)}+j(x)^{p'(x)}\right)
dx\right),$$ for all $t>0$, where $C$ is a constant depending only
on $\alpha$, $\gamma$ and $p(\cdot)$. \label{lem:cris}
\end{lem}

\begin{proof}
Take $\varphi\in \convex$ in the variational inequality
$\entropycond$ to obtain
\begin{eqnarray}\label{k01}
\int_{\{|u-\varphi|\leq t\}}a(x,\nabla u)\cdot \nabla (u-\varphi) \
dx \leq \int_\Om f\: T_t(u-\varphi)\ dx\leq \|f\|_1t,
\end{eqnarray}
for all $t>0$. On the other hand, using assumptions
\eqref{assumption1}--\eqref{assumption2} and Young's inequality, we
have, for all $t>0$,
\begin{eqnarray}\label{k02}
&\ &\hspace{-1cm}\int_{\{|u-\varphi|\leq t\}}a(x,\nabla u)\cdot
\nabla (u-\varphi) \
dx\nonumber\\
&\ &\geq \alpha\int_{\{|u-\varphi|\leq t\}}|\nabla u|^{p(x)}\ dx-
\gamma\int_{\{|u-\varphi|\leq t\}}(j(x)+|\nabla
u|^{p(x)-1})|\nabla\varphi|\ dx\nonumber\\
&\ &\geq \frac{\alpha}{2}\int_{\{|u-\varphi|\leq t\}}|\nabla
u|^{p(x)}\
dx-C\int_\Om\left(|\nabla\varphi|^{p(x)}+j(x)^{p'(x)}\right)\ dx,
\end{eqnarray}
where $C$, here and in the rest of the proof, is a constant
depending only on $\alpha$,
$\gamma$,
and $p(\cdot)$. Now, from \eqref{k01} and \eqref{k02}, we obtain
$$
\int_{\{|u-\varphi|\leq t\}}|\nabla u|^{p(x)}\ dx\leq
\frac{2\|f\|_1t}{\alpha}+C\int_\Om\left(|\nabla
\varphi|^{p(x)}+|j(x)|^{p'(x)}\right)\ dx,
$$
for all $t>0$. Replacing $t$ with $t+\|\varphi\|_\infty$ in the last
inequality, we get
\begin{eqnarray*}
\int_{\{|u|\leq t\}}|\nabla u|^{p(x)}\
dx&\leq&\int_{\{|u-\varphi|\leq
t+\|\varphi\|_\infty\}}|\nabla u|^{p(x)}\ dx\nonumber\\
&\hspace{-4.7cm}\leq&\hspace{-2.3cm}
C\left((t+\|\varphi\|_\infty)\|f\|_1+
\int_\Om\left(|\nabla\varphi|^{p(x)}+j(x)^{p'(x)}\right)\ dx\right),
\end{eqnarray*}
for all $t>0$.
\end{proof}

In the next result we prove \textit{a priori} estimates for an
entropy solution of $\entropycond$ in Marcinkiewicz spaces with
variable exponent. The proof is based on Lemma \ref{lem:cris} and
Sobolev inequality.

\begin{prop}\label{prop3:3}
Assume \eqref{assumption1}--\eqref{assumptionf} and let
$\varphi\in\convex$. If $u$ is an entropy solution of the
variational inequality $\entropycond$ then the following assertions
hold:

\begin{itemize}

\item[{\rm(}i{\rm)}]
There exists a positive constant $M$, depending only on $\alpha$,
$\gamma$, $N$, $p(\cdot)$, and $\Om$, such that

$$\hspace*{-7.5cm}\int_{\{|u|> t\}}\hspace{-0.5cm}t^{p^*(x)/\overline{p'}} dx$$
$$\leq
M\left((1+\|\varphi\|_\infty)\|f\|_1+\int_\Om\left(|j(x)|^{p'(x)}
+|\nabla\varphi|^{p(x)}\right)
dx+1\right)^{\overline{p^*}/\underline{p'}},
$$
for all $t>0$.

\item[{\rm(}ii{\rm)}] If there exists a positive constant $M$ such that
\begin{equation}
\int_{\{|u|> t\}}t^{q(x)}\ dx\leq M,\quad \textrm{for all }t>0,
\label{cris1}
\end{equation}
then $|\nabla u|^{r(\cdot)}\in M^{q(\cdot)}(\Om)$, where
$r(\cdot):=p(\cdot)/(q(\cdot)+1)$. Moreover, there exists a constant
$C$, depending only on $\alpha$, $\gamma$, and $p(\cdot)$, such that
$$\hspace*{-8.5cm} \int_{\left\{ |\nabla u|^{r(\cdot)} > t
\right\}}\hspace{-0.5cm}t^{q(x)}\ dx$$
$$\leq C\left((1+\|\varphi\|_\infty)\|f\|_1
+\int_\Om\left(|\nabla\varphi|^{p(x)}+|j(x)|^{p'(x)}\right)dx\right)+M+|\Om|,
$$
for all $t>0$.
\end{itemize}
\label{lem:u_nbounds}
\end{prop}


\begin{proof}
(\textit{i}) We proceed as in the proof of Proposition~3.2 in
\cite{SU06}, sketching here only the main steps (we refer to
\cite{SU06} for a complete account of the details). {F}rom Lemma
\ref{lem:cris}, we have
\begin{equation}\label{pepe4bis}
\frac{1}{t}\int_\Om |\nabla T_t(u)|^{p(x)}\ dx\leq
M_1+\frac{M_2}{t},
\end{equation}
for all $t>0$, where $M_1:=C_1\|f\|_1$ and
$$M_2:=C_1\left(\|f\|_1\|\varphi\|_\infty+
\int_\Om\left(|\nabla\varphi|^{p(x)}+j(x)^{p'(x)}\right)\ dx\right),
$$
for a constant $C_1$, depending only on $\alpha$, $\gamma$ and
$p(\cdot)$. On the other hand, using Lemma 2.3 in \cite{SU06} and
Sobolev's inequality, we estimate
\begin{eqnarray}
\int_{\{|u|> t\}}t^{p^*(x)/\overline{p'}}\ dx \leq C_2\left(
\int_\Om t^{p(x)/\overline{p'}+1-p(x)}\frac{1}{t}|\nabla
T_t(u)|^{p(x)}\ dx
+1\right)^{\overline{p^*}/\underline{p}}\!\!\!\!\!\!, \label{pepe1}
\end{eqnarray}
where $C_2$ is a constant depending only on $N$, $p(\cdot)$, and
$\Omega$.

Noting that $t^{p(x)/\overline{p'}+1-p(x)}\leq 1$ for all $t\geq 1$,
and using \eqref{pepe4bis}, we obtain
$$
\int_{\{|u|> t\}}t^{p^*(x)/\overline{p'}}\ dx \leq C_2\left(
M_1+M_2+1\right)^{\overline{p^*}/\underline{p}},
$$
for all $t\geq1$. For $0<t<1$, we have
$$\int_{\{|u|> t\}}t^{p^*(x)/\overline{p'}}\ dx\leq |\Omega|.$$
Combining both estimates and using the definition of $M_1$ and
$M_2$, we prove the result after simple estimates.

(\textit{ii}) Using \eqref{cris1}, the definition of $r(\cdot)$, and
\eqref{pepe4bis}, we have
\begin{eqnarray*}
\int_{\{ |\nabla u|^{r(x)} > t \} }t^{q(x)}\ dx & \leq & \int_{\{
|\nabla u|^{r(x)} > t \}\cap\{|u|\leq t\}}
t^{q(x)}\ dx+ \int_{\left\{ |u|> t \right\}}t^{q(x)}\ dx\\
&\hspace{-7cm} \leq &\hspace{-3.5cm} \int_{\{|u|\leq
t\}}t^{q(x)}\left(\frac{|\nabla
u|^{r(x)}}{t}\right)^{p(x)/r(x)} dx+M\\
&\hspace{-7cm} = &\hspace{-3.5cm}\frac{1}{t}\int_{\{|u|\leq
t\}}|\nabla
T_t(u)|^{p(x)}\ dx+M\\
&\hspace{-7cm} \leq
&\hspace{-3.5cm}C\left((1+\|\varphi\|_\infty)\|f\|_1
+\int_\Om\left(|\nabla\varphi|^{p(x)}+|j(x)|^{p'(x)}\right)dx\right)
+M,
\end{eqnarray*}
for all $t\geq 1$, where $C$ is a constant depending only on
$\alpha$ and $p(\cdot)$. Noting that
$$\int_{\{ |\nabla u|^{r(x)} > t \} }t^{q(x)}\ dx\leq
|\Om|,\qquad \textrm{for all }t\leq 1,$$ we conclude the proof.
\end{proof}

Using Proposition \ref{lem:u_nbounds} and Proposition~2.5 in
\cite{SU06} one obtains the following result (see the proofs of
Corollaries 3.5 and 3.7 in \cite{SU06}).

\begin{cor}\label{cor1}
Assume \eqref{assumption1}--\eqref{assumptionf}. Let
\begin{equation}\label{q_0andq_1}
q_0(\cdot) =\frac{p^*(\cdot)}{\overline{p'}}\quad \textrm{and}\quad
q_1(\cdot) =\frac{q_0(\cdot)}{q_0(\cdot)+1}p(\cdot).
\end{equation}
If $u$ is an entropy solution of the variational inequality
$\entropycond$, then there exists a constant $C$, which is
independent of $u$, such that
\begin{equation}\label{L^q_bound}
\int_\Om |u|^{q(x)} \ dx \leq C,\quad \textrm{for all }0 \ll
q(\cdot) \ll q_0(\cdot),
\end{equation}
and
\begin{equation}\label{grad_bound}
\int_\Om |\nabla u|^{q(x)} \ dx \leq C, \quad \textrm{for all }0 \ll
q(\cdot) \ll q_1(\cdot).
\end{equation}
In particular, $|u|^{q(\cdot)} \in L^1 (\Om)$, for all $q(\cdot)$
such that $0 \ll q(\cdot) \ll q_0(\cdot)$, and $|\nabla
u|^{q(\cdot)} \in L^{1} (\Om)$, for all $q(\cdot)$ such that $0 \ll
q(\cdot) \ll q_1(\cdot)$.
\end{cor}

%
%

\section{Existence and uniqueness of entropy solutions}
\label{section3}

In this section we prove the existence and uniqueness of an entropy
solution to the obstacle problem $\entropycond$. We also prove the
continuous dependence of the solution with respect to the
right--hand side $f$ and the obstacle $\psi$.

We start by proving that a sequence $\{u_n\}_n$ of entropy solutions
of the obstacle problems $\entropycondn$ converges in measure to a
measurable function $u$. We also show that the sequence of weak
gradients $\{\nabla u_n\}_n$ converges in measure to $\nabla u$, the
weak gradient of $u$. Finally, we prove some regularity properties
using Proposition \ref{lem:u_nbounds} and Corollary \ref{cor1}.

\begin{prop}\label{prop:key}
Let $\{f_n,\psi_n\}_n$ be a sequence in $L^1(\Om)\times
W^{1,p(\cdot)}(\Om)$. Assume
\eqref{assumption1}--\eqref{assumptionf} and that ${\psi_n}^+
\in\Wpx \cap L^\infty (\Om)$, for all $n$. Let $u_n$ be an entropy
solution of the obstacle problem ${\rm(}1.5{\rm})_{f_n,\psi_n}$. If
\begin{equation}\label{ffn}
f_n  \longrightarrow f \quad \textrm{in }L^1(\Om) \qquad and \qquad
\psi_n \longrightarrow \psi \quad \textrm{in }W^{1,p(\cdot)}(\Om),
\end{equation}
then the following assertions hold:

\begin{description}

\item[{\rm(}i{\rm)}] There exists a measurable function $u$ such that
$u_n\rightarrow u$ in measure.

\item[{\rm(}ii{\rm)}] $\nabla u_n$ converges in measure to $\nabla u$, the weak gradient of $u$.
\end{description}

\end{prop}
\begin{proof}
Let $\varphi\in\convex$, \textit{e.g.} $\varphi=\psi^+$, and note
that $\varphi_n:=\varphi+(\psi_n-\varphi)^+\in L^\infty(\Om)$ since
$\varphi\in L^\infty(\Om)$ and $\psi_n$ is bounded above (see
Remark~\ref{rem:thm1}). In particular, $\varphi_n\in\convexn$.
Moreover, by \eqref{ffn}, there exists a constant $C$, independent
of $n$, such that
\begin{equation}\label{lalarito}
\|f_n\|_1\leq C(\|f\|_1+1),\qquad \|\varphi_n\|_\infty \leq
C\left(\|\varphi\|_\infty+1\right),
\end{equation}
and
\begin{equation}\label{lalarito2}
\int_\Om|\nabla\varphi_n|^{p(x)}\ dx\leq
C\left(\int_\Om|\nabla\varphi|^{p(x)}\ dx+1\right),\quad \mbox{for
all }n.
\end{equation}

\medskip

\noindent (\textit{i}) Let $s$, $t$, and $\varepsilon$ be positive
numbers. Noting that
\begin{eqnarray}
\textrm{meas}\: \{|u_n-u_m|>s\}&\leq &\textrm{meas}\: \{|u_n|>t\}
+\textrm{meas}\:\{|u_m|>t\}\nonumber\\&+&
\textrm{meas}\:\{|T_t(u_n)-T_t(u_m)|>s\}, \label{measu1}
\end{eqnarray}
from Proposition \ref{lem:u_nbounds}(\textit{i}) and
\eqref{lalarito}--\eqref{lalarito2}, we can choose
$t=t(\varepsilon)$ such that
$\textrm{meas}\:\{|u_n|>t\}<\varepsilon/3$ and
$\textrm{meas}\:\{|u_m|>t\}<\varepsilon/3$. On the other hand, from
Lemma \ref{lem:cris} applied to $u_n$ and
\eqref{lalarito}--\eqref{lalarito2}, we obtain
$$\begin{array}{ll}
\displaystyle\int_\Om|\nabla T_t(u_n)|^{p(x)} dx &\displaystyle\leq
C\Big((t+\|\varphi\|_\infty+1)(\|f\|_1+1)\\
&\displaystyle+
\int_\Om\left(|\nabla\varphi|^{p(x)}+j(x)^{p'(x)}\right) dx+1\Big),
\end{array}$$ for all  $t>0$, where $C$ is a constant
depending only on $\alpha$, $\gamma$ and $p(\cdot)$. Therefore, we
can assume, by Sobolev embedding, that $\left\{
T_{t}(u_n)\right\}_n$ is a Cauchy sequence in $\Lqx$, for all $1\leq
q(\cdot)\ll p^*(\cdot)$. Consequently, there exists a measurable
function $u$ such that
$$T_{t}(u_n)\longrightarrow T_{t}(u),\quad \textrm{in }\Lqx \ \textrm{and a.e. in } \Omega.$$
Thus,
$$\textrm{meas}\: \{|T_t(u_n)-T_t(u_m)|>s\} \leq \int_\Om \left(
\frac{\left| T_{t}(u_n)-T_{t}(u_m) \right|}{s} \right)^{q(x)}\ dx <
\frac{\epsilon}{3}$$ for all $n,m\geq n_0(s,\epsilon)$. Finally,
from \eqref{measu1}, we obtain
$$\textrm{meas}\: \{|u_n-u_m|>s\}< \epsilon, \quad \textrm{for all}\ n,m\geq n_0(s,\epsilon),$$
\textit{i.e.}, $\{u_n\}_n$ is a Cauchy sequence in measure. The
assertion follows.

\medskip

\noindent The proof of (ii) is entirely similar to the corresponding
one in Proposition~5.3 of \cite{SU06}. We omit the details.
\end{proof}

At this point, we prove Theorem \ref{thm2} using Proposition
\ref{prop:key}.


\medskip
\noindent \emph{Proof of Theorem} {\rm \ref{thm2}}. Let
$\varphi\in\convex$ and define
$\varphi_n:=\varphi+(\psi_n-\varphi)^+$. Note that
$\varphi_n\in\convexn$ and that $\varphi_n$ converges strongly to
$\varphi$ in $\Wpx$, due to \eqref{convergence:fn}. Taking
$\varphi_n$ as a test function in $\entropycondn$, we obtain
$$\int_\Om a(x,\nabla u_n)\cdot \nabla T_t(u_n-\varphi_n)\ dx\leq
\int_\Om f_n(x) T_t(u_n-\varphi_n)\ dx.$$ Next, we pass to the limit
in the previous inequality.

Concerning the right-hand side, the convergence is obvious since
$f_n$ converges to $f$, strongly in $L^1(\Om)$, and
$T_t(u_n-\varphi_n)$ converges to $T_t(u-\varphi)$, weakly--$\ast$
in $L^\infty$ and a.e. in $\Om$. To deal with the left-hand side we
write it as
\begin{equation}\label{lhs}
 \int_{\{|u_n-\varphi_n|\leq t\}} a(x,\nabla u_n)\cdot \nabla u_n \
dx- \int_{\{|u_n-\varphi_n|\leq t\}} a(x,\nabla u_n)\cdot \nabla
\varphi_n \ dx
\end{equation}
and note that $\{|u_n-\varphi_n|\leq t\}$ is a subset of
$\{|u_n|\leq t+C(\|\varphi\|_\infty+1)\}$, where $C$ is a constant
that does not depend on $n$ (see \eqref{lalarito}). Hence, taking
$s=t+C(\|\varphi\|_\infty+1)$, we rewrite the second integral in
\eqref{lhs} as
$$\int_{\{|u_n-\varphi_n|\leq t\}} a(x,\nabla T_s(u_n))\cdot \nabla
\varphi_n\ dx.$$ Since $a(x,\nabla T_s(u_n))$ is uniformly bounded
in $(L^{p'(\cdot)}(\Om))^N$ (by assumption \eqref{assumption2} and
Lemma \ref{lem:cris}), it converges weakly to $a(x,\nabla T_s(u))$
in $(L^{p'(\cdot)}(\Om))^N$, due to Proposition
\ref{prop:key}(\textit{ii}). Therefore the last integral converges
to
$$\int_{\{|u-\varphi|\leq t\}} a(x,\nabla u))\cdot \nabla \varphi \
dx.$$ The first integral in \eqref{lhs} is nonnegative, by
\eqref{assumption1}, and it converges a.e. by Proposition
\ref{prop:key}. It follows from Fatou's lemma that
$$\int_{\{|u-\varphi|\leq t\}} a(x,\nabla u)\cdot \nabla u \ dx\leq
\liminf_{n\rightarrow +\infty}\int_{\{|u_n-\varphi_n|\leq t\}}
a(x,\nabla u_n)\cdot \nabla u_n \ dx.$$

Gathering results, we obtain
$$\int_\Om a(x,\nabla u)\cdot \nabla T_t(u-\varphi) \ dx\leq \int_\Om
f\: T_t(u-\varphi)\ dx,$$ \textit{i.e.}, $u$ is an entropy solution
of $\entropycond$.\qed

\medskip

Finally, we prove Theorem \ref{thm1}, as an application of Theorem
\ref{thm2}.

\medskip


\noindent \emph{Proof of Theorem} {\rm \ref{thm1}}. Let us consider
the sequence of approximated obstacle problems $\entropycondfn$,
where $\{f_n\}_n$ is a sequence of bounded functions strongly
converging to $f$ in $L^1(\Om)$. It is straightforward, from
classical results (see \cite{L69,KinStam80}), to prove the existence
of a unique solution $u_n\in \Wpx$ of the obstacle problem
$\entropycondfn$. Noting that a weak energy solution is also an
entropy solution, we may apply Theorem \ref{thm2} to obtain that
$u_n$ converges to a measurable function $u$ which is an entropy
solution of the limit obstacle problem $\entropycond$. Now, the
regularity stated in the theorem follows immediately from Corollary
\ref{cor1}.

Finally, we prove the uniqueness. Let $u$ and $v$ be entropy
solutions of $\entropycond$. Since $\psi^+ \in\Wpx \cap L^\infty
(\Om)$ and $\psi\leq \|\psi^+\|_\infty$, $T_h u$ and $T_h v$ belong
to the convex set ${\mathcal K}_\psi$ for $h>0$ large enough. Now,
we proceed as in the proof of Theorem~4.1 in \cite{SU06}. We write
the variational inequality $\entropycond$ corresponding to the
solution $u$, with $T_hv$ as test function, and to the solution $v$,
with $T_hu$ as test function. Upon addition, we get
$$\int_{\{|u-T_hv|\leq t\}}a(x,\nabla u)\cdot \nabla (u-T_hv)\ dx
 +\int_{\{|v-T_hu|\leq t\}}a(x,\nabla v)\cdot \nabla (v-T_hu)\ dx$$
$$\leq \int_\Om f\: \Big(T_t(u-T_hv)+T_t(v-T_hu)\Big)\ dx.$$
We let $h$ go to infinity in this inequality. By Proposition
\ref{lem:u_nbounds}(\textit{i}), it is easy to prove that the
right-hand side tends to zero. Moreover, using assumptions
\eqref{assumption1}--\eqref{assumption2}, H\"older's inequality, and
Proposition~\ref{lem:u_nbounds}(\textit{ii}) to study the left-hand
side, we obtain
$$\int_{\{|u-v|\leq t\}} \left( a(x,\nabla u)-a(x,\nabla v)
\right)\cdot \nabla (u-v)\ dx\leq 0, \quad \mbox{for all} \ t>0.$$
By assumption \eqref{assumption3}, we conclude that $\nabla u=\nabla
v$, a.e. in $\Om$, and hence, from Poincar\'e's inequality, it
follows that $u=v$, a.e. in $\Om$.\qed

\section{Lewy--Stampacchia inequalities and stability of the coincidence set}
\label{section4}

The aim of this section is to prove the Lewy--Stampacchia
inequalities and the resulting properties stated in Section
\ref{section1.1}.

In order to prove Theorem \ref{thm3}, we consider a sequence of
approximated obstacle problems for which the abstract theory
developed in \cite{M,AtP} applies. Once we have the
Lewy--Stampacchia inequalities for the approximated problems, we may
pass to the limit using
the following proposition.
\begin{prop}\label{L1convergence}
Assume $p(\cdot)-1\ll q_1(\cdot)$. Under the assumptions of
Proposition {\rm \ref{prop:key}} the following assertions hold:
\begin{itemize}
\item[{\rm(}i{\rm)}] $a(x,\nabla u_n)$ converges to $a(x,\nabla u)$, strongly in
$L^1(\Om)$.

\item[{\rm(}ii{\rm)}] $a(x,\nabla u)\in \Lqx$, for some $1\leq q(\cdot)$.

\item[{\rm(}iii{\rm)}] $u$ and $\nabla u$ satisfy \eqref{L^q_bound} and
\eqref{grad_bound}.
\end{itemize}
\end{prop}
\begin{proof}
We omit the proof, since it is analogous to the proof of
Proposition~5.5 in \cite{SU06}.
\end{proof}

\begin{rem}
{\rm As pointed out in \cite{SU06}, assumption $p(\cdot)-1\ll
q_1(\cdot)$, which is obviously satisfied for $p$ constant, is
equivalent to the condition
\begin{equation}
\frac{Np^\prime (\cdot)}{N-p(\cdot)}  \gg \overline{p^\prime} .
\label{xs}
\end{equation}
The analysis of the behaviour of the function on the left-hand side
of this inequality leads to the following conclusions:

\medskip

\begin{itemize}

\item[(i)] if $\overline{p} < \sqrt{N}$ then \eqref{xs} is satisfied for
any function $p(\cdot)$ such that
$$\frac{1}{\underline{p}} - \frac{1}{\overline{p}} < \frac{\overline{p} -1 }{N};$$

\medskip

\item[(ii)] if $\underline{p} \leq \sqrt{N} \leq \overline{p}$ then \eqref{xs} is satisfied for
any function $p(\cdot)$ such that
$$\underline{p} > \frac{N}{2\sqrt{N}-1} ;$$

\medskip

\item[(iii)] if $\underline{p} > \sqrt{N}$ then \eqref{xs} is satisfied for
any function $p(\cdot)$.

\end{itemize}

\medskip

\noindent The condition in case (i) only holds if $\underline{p}$ is
\textit{close} to $\overline{p}$, so it forces a modest variation in
the field of values of $p(\cdot)$. }
\label{new_rem}
\end{rem}

\medskip

Now, we are able to prove Theorem \ref{thm3}.

\medskip

\noindent \emph{Proof of Theorem} {\rm \ref{thm3}}. Consider a
sequence $\{f_n\}_n$ of $L^\infty (\Om)$ functions such that $f_n
\rightarrow f$ in $L^1(\Om)$. Let $u_n\in\Wpx$ be the unique weak
energy solution of the obstacle problem
$$ u_n \in {\mathcal K}_\psi \ : \ \left\langle {\mathcal A} u_n - f_n, v-u_n \right\rangle
\geq 0, \quad \forall v \in {\mathcal K}_\psi .$$ Since $V:=\Wpx$ is
a reflexive Banach space and ${\mathcal A}:V \rightarrow V^\prime$
is strictly $T$-monotone, it follows from the abstract theory
developed in \cite{M} that
$$f_n \leq {\mathcal A} u_n \leq f_n + ({\mathcal A}\psi-f_n)^+ \quad \textrm{in} \ V^\prime.$$
In particular, these inequalities hold in the sense of
distributions.

Let $0 \leq \varphi \in \mathcal{D} (\Om)$; then
$$
\int_\Om f_n \varphi\ dx \leq \int_\Om a(x, \nabla u_n) \cdot \nabla
\varphi\ dx \leq \int_\Om \left[f_n + ({\mathcal
A}\psi-f_n)^+\right] \varphi\ dx.
$$
We can pass to the limit in this expression using the facts that
$f_n \rightarrow f$ in $L^1(\Om)$ and $a(x, \nabla u_n) \rightarrow
a(x, \nabla u)$ in $L^1(\Om)$ (see
Proposition \ref{L1convergence}(\textit{i})),
and obtain
$$f \leq {\mathcal A}u \leq f+({\mathcal A}\psi-f)^+ \quad \mbox{in}\ \mathcal{D}^\prime (\Om).$$
Finally, since $f$ and $f+({\mathcal A}\psi-f)^+$ are $L^1(\Om)$
functions, we conclude that also ${\mathcal A}u \in L^1(\Om)$ and
\eqref{LeviStam_ineq} follows.\qed

In order to prove Theorem \ref{thm4} we need two preliminary lemmas.

\begin{lem}\label{lem2}
Let $w_i$ be measurable functions such that $T_t(w_i) \in \Wpx$, for
all $t>0$, $a(x, \nabla w_i) \in \left[ L^1(\Om) \right]^N$, and
${\mathcal A} w_i \in L^1 (\Om)$, for $i=1,2$. Then
\begin{equation}
{\mathcal A}w_1 = {\mathcal A}w_2 \quad \mbox{a.e. in} \ \: \{
w_1=w_2 \} . \label{lem}
\end{equation}
\end{lem}

\begin{proof}
Let
$$\pmb{L}^1_{\nabla} (\Om) = \left\{ \pmb{\xi} \in \left[ L^1(\Om) \right]^N \ : \
{\rm div}\ \pmb{\xi} \in L^1(\Om)\right\} .$$ Since $\left[
C^1(\overline{\Om}) \right]^N$ is dense in $\pmb{L}^1_{\nabla}
(\Om)$ for the graph norm, it follows from the arguments in
\textit{Lemmata} A3 and A4 of \cite[pages 52--53]{KinStam80} that
the following property holds in $\pmb{L}^1_{\nabla} (\Om)$:
$${\rm div}\ \pmb{\xi} = 0  \quad \mbox{a.e. in} \ \: \{ \pmb{\xi} = \pmb{0} \} . $$
Due to the assumptions, $a(x, \nabla w_1) - a(x, \nabla w_2) \in
\pmb{L}^1_{\nabla} (\Om)$, so we have
\begin{equation}
{\mathcal A}w_1 = {\mathcal A}w_2  \quad \mbox{a.e. in} \ \: \left\{
a(x, \nabla w_1) = a(x, \nabla w_2) \right\} . \label{help}
\end{equation}
Finally, it is standard that
$$\nabla T_t (w_1) = \nabla T_t (w_2) \quad \mbox{a.e. in} \ \: \left\{  w_1 = w_2 \right\} ,$$
for any $t>0$, so the weak gradients $\nabla w_1$ and $\nabla w_2$
coincide in $\left\{  w_1 = w_2 \right\}$ and the conclusion follows
from \eqref{help}.
\end{proof}

The other lemma requires a definition of the coincidence set for the
obstacle problem, which poses a difficulty in face of the available
regularity for the solution and the obstacle. Indeed, if $u$ and
$\psi$ are continuous functions, the coincidence set is defined as
the closed subset of $\Omega$
$$
\left\{ x \in \Omega \ : \ u(x) = \psi (x) \right\} = \left( u-\psi
\right)^{-1} \left( \{ 0 \} \right)\ ,
$$
and this definition is unambiguous. But, in general, the entropy
solution is not necessarily continuous, and we are not making that
assumption for the obstacle either. So we need to interpret the
coincidence set in a different and more elaborate sense.

We first define the \textit{non--coincidence set} $\{ u > \psi\}$.
Since $\psi$ is bounded above (\textit{cf.} Remark \ref{rem:thm1}),
we can take $s> \sup_\Om \psi$. The function $T_s(u)$ belongs to
$\Wpx$, by the definition of entropy solution. Then
$$\{ u > \psi\} := \left\{ x \in \Omega \ : \ \left( T_s(u)-\psi
\right) (x) > 0 \ \mbox{in the sense of} \
W^{1,p(\cdot)}(\Om)\right\}\ .$$ Given $w \in W^{1,p(\cdot)}(\Om)$,
we say that $w(x)>0$ in the sense of $W^{1,p(\cdot)}(\Om)$ if there
exists a neighborhood of $x$, $N_x \subset \Omega$, and a
nonnegative function $\zeta \in W^{1,\infty} (N_x)$, such that
$\zeta (x)
>0$ and $w \geq \zeta$ a.e. in $N_x$. The definition is clearly
independent of the choice of $s$ and it turns out that $\{ u >
\psi\}$ is necessarily an open subset of $\Omega$. We then define
the \textit{coincidence set} as
$$
\{ u = \psi\} := \Omega \setminus \{ u > \psi\}\ .
$$

\begin{lem}\label{lem1}
Assume \eqref{assumption1}--\eqref{assumptionf} and $p(\cdot)-1\ll
q_1(\cdot)$. The entropy solution of the obstacle problem
$\entropycond$ solves
\begin{equation}
{\mathcal A}u =f \ , \quad \mbox{a.e. in } \{ u > \psi\}\ .
\label{eq1}
\end{equation}
\end{lem}

\begin{proof}
To simplify, let us denote $\Lambda = \{ u > \psi \}$, which is an
open subset of $\Omega$. Let $\varphi \in \mathcal{D} (\Lambda)$.
Let $h > \sup_\Om\psi$ and choose $\varepsilon>0$ small enough such
that
$$v=T_h(u) \pm \varepsilon \varphi \ \in \ \convex.$$
Taking $v$ as a test function in $\entropycond$, we obtain
$$\int_\Om a(x,\nabla u) \cdot \nabla T_t \left( T_h(u) \pm
\varepsilon \varphi-u \right)\ dx \geq \int_\Om f \: T_t \left(
T_h(u) \pm \varepsilon \varphi-u \right)\ dx.$$ From
\eqref{assumption1}, it follows that
$$\pm \varepsilon \int_{ \left\{ \, |T_h(u) \pm \varepsilon \varphi-u|\leq t \right\}}
a(x,\nabla u) \cdot \nabla \varphi\ dx \geq \int_\Om f \: T_t \left(
T_h(u) \pm \varepsilon \varphi-u \right)\ dx.$$ Choosing
$t>\varepsilon \| \varphi \|_\infty$ and letting $h \rightarrow
\infty$ (using Proposition \ref{L1convergence} (i)), we obtain
$$\pm \varepsilon \int_{ \Lambda}
a(x,\nabla u) \cdot \nabla \varphi\ dx \geq \pm \varepsilon \int_{
\Lambda} f \varphi\ dx,$$ and, hence, we conclude that
$${\mathcal A}u= - \, \textrm{div} \ a(x,\nabla u) = f \qquad \textrm{in} \quad \mathcal{D}^\prime (\Lambda)$$
and the result follows.
\end{proof}

We prove Theorem \ref{thm4} as a consequence of \textit{Lemmata}
\ref{lem2} and \ref{lem1}.

\medskip
\noindent \emph{Proof of Theorem} {\rm \ref{thm4}}. By the previous
lemma, we have ${\mathcal A}u=f$, a.e. in $\{ u > \psi \}$. The
result follows from the fact that ${\mathcal A}u={\mathcal A}\psi$,
a.e. in $\{ u=\psi\}$, which is a consequence of Lemma \ref{lem2},
since ${\mathcal A}u\in L^1(\Om)$ by Theorem \ref{thm3}.\qed

Using Theorems \ref{thm2} and \ref{thm4} we prove the convergence of
a sequence of coincidence sets to the coincidence set of the limit.

\medskip
\noindent \emph{Proof of Theorem} {\rm \ref{thm5}}. Let $u_n$ and
$u$ be the entropy solutions of the obstacle problems
$\entropycondn$ and $\entropycond$, respectively. By Theorem
\ref{thm2}, $u_n$ converges to $u$ in measure, and hence, a.e. in
$\Om$. Moreover, by Theorem \ref{thm4}, and denoting
$\chi_{_n}=\chi_{\{u_n=\psi_n\}}$, $u_n$ satisfies
\begin{equation}\label{pepaflores}
{\mathcal A}u_n-({\mathcal A}\psi_n-f_n)\chi_{_n}=f_n,\quad
\textrm{a.e. in }\Om, \textrm{ for all }n.
\end{equation}
Since $0\leq \chi_{_n}\leq 1$, there exists a subsequence (still
denoted by $\chi_{_n}$) and a function $\chi\in L^\infty(\Om)$, such
that
$$\chi_{_n}\rightharpoonup \chi \quad \textrm{weakly}-\ast \ \textrm{ in
}L^\infty(\Om).$$ Hence, since ${\mathcal A}\psi_n\rightarrow
{\mathcal A}\psi$ and $f_n\rightarrow f$, strongly in $L^1(\Om)$,
taking the limit in \eqref{pepaflores} we obtain
$${\mathcal A}u-({\mathcal A}\psi-f)\chi=f, \quad \textrm{a.e. in }\Om.$$

On the other hand, by Theorem \ref{thm4}, $u$ also satisfies the
previous identity with $\chi$ replaced by $\chi_{\{u=\psi\}}$.
Therefore, using ${\mathcal A}\psi\neq f,$ a.e. in $\Om$, the whole
sequence $\chi_{_n}$ converges to the characteristic function
$\chi_{\{u=\psi\}}$ and satisfies \eqref{limit_coincidence}. The
theorem is proved.\qed

Finally, we prove Theorem \ref{thm6} using again Proposition
\ref{prop:key} and the Lewy--Stampacchia inequalities.

\medskip
\noindent \emph{Proof of Theorem} {\rm \ref{thm6}}. First, we claim
that
\begin{equation}\label{graph}
\int_\Om ({\mathcal A}u_1-{\mathcal A}u_2)\: \varphi\ dx \geq 0,
\quad \forall \ \varphi\in\sigma(u_1(x)-u_2(x)).
\end{equation}
Here $\sigma$ de\-notes the maximal monotone graph associated to the
sign function (\textit{i.e.}, $\sigma=\partial r$, $r(t)=|t|$).

Indeed, let $\{f_i^n\}_n$ be a sequence of bounded functions
strongly converging in $L^1(\Om)$ to $f_i$ ($i=1,2$), and let $u_i^n
\in \Wpx$ be the corresponding weak energy solutions of
$\entropycondfin$. Let $(\sigma_\varepsilon)_{\varepsilon>0}$ be a
sequence of smooth functions satisfying $\sigma_\varepsilon(0)=0$,
$|\sigma_\varepsilon(t)|\leq 1$ and $\sigma'_\varepsilon(t)\geq 0$,
for all $t\in\R$, such that $\sigma_\varepsilon(t)\rightarrow
\textrm{sign}\, (t)$ as $\varepsilon\downarrow 0$. Integration by
parts and the use of assumption \eqref{assumption3} yield the
inequality
$$\int_\Om ({\mathcal A}u_1^n-{\mathcal A}u_2^n)\: \sigma_\varepsilon(u_1^n-u_2^n)\ dx$$
\begin{equation}
= \int_\Om \left( a(x,\nabla u_1^n)-a(x,\nabla u_2^n) \right) \cdot
\nabla(u_1^n-u_2^n)\: \sigma'_\varepsilon(u_1^n-u_2^n)\ dx \geq 0.
\label{limi}
\end{equation}
We now pass to the limit as $n \rightarrow \infty$. To start with,
we have (for a subsequence, relabeled if need be)
$${\mathcal A}u_1^n-{\mathcal A}u_2^n \rightharpoonup {\mathcal A}u_1-{\mathcal A}u_2, \quad \mbox{weakly in }
L^1(\Om).$$ This follows from Dunford-Pettis Theorem (the hypothesis
of which are satisfied due to the Lewy--Stampacchia inequalities),
and the fact that the convergence holds in the sense of
distributions since, by Proposition \ref{L1convergence}(\textit{i}),
$$a(x,\nabla u_1^n)-a(x,\nabla
u_2^n) \longrightarrow a(x,\nabla u_1)-a(x,\nabla u_2), \quad
\mbox{in } L^1(\Om).$$ On the other hand, by Proposition
\ref{prop:key}(\textit{i}),
$$\sigma_\varepsilon(u_1^n-u_2^n) \longrightarrow \sigma_\varepsilon(u_1-u_2),
\quad \mbox{a.e in } \Om .$$ Fix an arbitrary $\delta >0$. Again
from the Lewy--Stampacchia inequalities, we can find $\nu
>0$ such that, for all $A \subset \Om$,
\begin{equation}
\textrm{meas}(A)<\nu \Longrightarrow \int_A |{\mathcal
A}u_1^n-{\mathcal A}u_2^n|\ dx< \frac{\, \delta}{\, 4} \: , \quad
\textrm{for all } \: n. \label{equi}
\end{equation}
By Egorov's Theorem, there exists a measurable subset $\omega
\subset \Omega$ such that
\begin{equation}
\textrm{meas}\left(\Om \setminus \omega \right) < \nu
\label{egorov1}
\end{equation}
and
\begin{equation}
\sigma_\varepsilon(u_1^n-u_2^n) \longrightarrow
\sigma_\varepsilon(u_1-u_2), \quad \mbox{uniformly in } \omega.
\label{egorov2}
\end{equation}
To lighten the notation, we put $F^n:={\mathcal A}u_1^n-{\mathcal
A}u_2^n$ and $G_\varepsilon^{n}:= \sigma_\varepsilon(u_1^n-u_2^n)-
\sigma_\varepsilon(u_1-u_2)$. Then,
\begin{eqnarray}
\left| \int_\Om F^n (x) \: G_\varepsilon^{n} (x)   \ dx \right| &
\leq & \left| \int_{\Om \setminus \omega} F^n (x) \:
G_\varepsilon^{n} (x)   \ dx \right| +
\left| \int_\omega F^n (x) \: G_\varepsilon^{n} (x)   \ dx \right| \nonumber\\
 & \leq & 2 \int_{\Om \setminus \omega} \left| F^n (x)  \right| \ dx  +  \int_\omega
\left| F^n (x)\right| \: \left| G_\varepsilon^{n} (x) \right|  \ dx \nonumber\\
 & \leq & 2 \: \frac{\, \delta}{\, 4} + \kappa \: \frac{\, \delta}{\, 2\kappa} \nonumber\\
 & = & \delta \, , \label{bd}
\end{eqnarray}
for all $n \geq n_0$, using \eqref{egorov1} and \eqref{equi} to
bound the first term and \eqref{egorov2} to bound the second. Here
$\kappa > 0$ is a constant (which exists due to the
Lewy--Stampacchia inequalities) such that
$$\int_\omega \left| F^n (x)\right| \ dx \leq \int_\Omega \left| {\mathcal A}u_1^n-{\mathcal A}u_2^n \right|
\ dx \leq \kappa, \quad \forall \: n.$$ Since $\delta >0$ is
arbitrary, we conclude from \eqref{bd} that
$$\int_\Om \left( {\mathcal A}u_1^n-{\mathcal A}u_2^n \right) \: \left[ \sigma_\varepsilon(u_1^n-u_2^n)-
\sigma_\varepsilon(u_1-u_2) \right]  \ dx \longrightarrow 0$$ so we
can pass to the limit in \eqref{limi} to obtain
$$\int_\Om ({\mathcal A}u_1-{\mathcal A}u_2) \: \sigma_\varepsilon(u_1-u_2)\ dx\geq 0.$$
Finally, letting $\varepsilon\downarrow 0$, we obtain \eqref{graph}
with $\varphi=\textrm{sign}\,(u_1-u_2)$. Since, by Lemma \ref{lem2},
$$({\mathcal A}u_1-{\mathcal A}u_2)\: \varphi = ({\mathcal A}u_1-{\mathcal A}u_2) \: \textrm{sign}\,(u_1-u_2),\quad
\textrm{a.e. }x\in\Om,$$ for all $\varphi\in\sigma(u_1-u_2)$, the
claim follows.

To conclude the proof, take $\varphi\in\sigma(u_1-u_2)$, defined by
$$\varphi:=\left\{
\begin{array}{cll}
-1&\textrm{in}&\{u_1<u_2\}\cup\{\xi_1<\xi_2\}\\
0&\textrm{on}&\{u_1=u_2\}\cap\{\xi_1=\xi_2\}\\
1&\textrm{in}&\{u_1>u_2\}\cup\{\xi_1>\xi_2\}.
\end{array} \right.$$
Multiplying
$$\xi_1-\xi_2=(f_1-f_2)-({\mathcal A}u_1-{\mathcal A}u_2)$$
by $\varphi$, integrating in $\Om$, and using \eqref{graph}, we
obtain
$$\int_\Om|\xi_1-\xi_2|\ dx = \int_\Om (\xi_1-\xi_2) \: \varphi \ dx
\leq \int_\Om(f_1-f_2)\: \varphi\ dx\leq \int_\Om|f_1-f_2|\ dx,$$
proving \eqref{L1}. Finally, by Theorem \ref{thm4}, we have
$\xi_i=(f_i-{\mathcal A}\psi)\chi_{\{u_i=\psi\}}$, for $i=1,2$.
Therefore
$$|\chi_{\{u_1=\psi\}}-\chi_{\{u_2=\psi\}}|\leq \frac{1}{\lambda}
|\xi_1-\xi_2|, \quad \textrm{a.e. in }D,$$ due to assumption
\eqref{assumptionX}. The theorem follows by integrating over
$D$.\qed

\bigskip

\noindent{\bf Acknowledgments.} The research of J.F. Rodrigues and
J.M. Urbano was partially supported by CMUC/FCT and Project
POCI/MAT/57546/2004.

The research of M. Sanch\'on was partially supported by CMUC/FCT and
MCYT grant MTM2005--07660--C02.


\begin{thebibliography}{99}

\bibitem{AM01} {Acerbi, E., Mingione, G.}: Regularity results for a class
of functionals with nonstandard growth. Arch. Ration. Mech. Anal.
\textbf{156} (2001), 121--140.

\bibitem{AA06}{Aharouch, L., Akdim, Y.}: Strongly nonlinear elliptic
unilateral problems without sign condition and $L^1$ data. J. Convex
Anal. {\bf 13} (2006), no. 1, 135--149.

\bibitem{Al} {Alkhutov, Yu.A.}: The Harnack inequality and the H\"older property of
solutions of nonlinear elliptic equations with a nonstandard growth
condition. Differential Equations \textbf{33} (1997), no. 12,
1653--1663.

\bibitem{ABFOT03} {Alvino, A., Boccardo, L., Ferone, V.,
Orsina, L., Trombetti, G.}: Existence results for nonlinear elliptic
equations with degenerate coercivity, Ann. Mat. Pura Appl.
\textbf{182} (2003), 53--79.

\bibitem{AR} Antontsev, S., Rodrigues, J.F.: On stationary thermorheological viscous
flows. Ann. Univ. Ferrara Sez. VII Sci. Mat. \textbf{52} (2006), no.
1, 19--36.

\bibitem{AS} Antontsev, S., Shmarev, S.: Elliptic equations with anisotropic
nonlinearity and nonstandard growth conditions. In: Handbook of
Differential Equations, Stationary Partial Differential Equations,
vol. 3, pp.1--100, Elsevier, 2006.

\bibitem{AtP} Attouch, H., Picard, C.: Probl\`emes variationnels et
th\'eorie du potentiel non lin\'eaire. Ann. Fac. Sci. Toulouse Math.
(5) \textbf{1} (1979), 89--136.

\bibitem{BBGGPV95} B\'enilan, Ph., Boccardo, L., Gallou\"et, T.,
Gariepy, R., Pierre, M., V\'azquez, J.L.: An $L\sp 1-$theory of
existence and uniqueness of solutions of nonlinear elliptic
equations. Ann. Scuola Norm. Sup. Pisa Cl. Sci. (4) {\bf 22} (1995),
241--273.

\bibitem{BC} Boccardo, L., Cirmi, G.R.: Existence and uniqueness of solution
of unilateral problems with $L\sp 1$ data. J. Convex Anal. {\bf 6}
(1999), 195--206.

\bibitem{BG} Boccardo, L., Gallou\"et, T.: Probl\`emes
unilat\'eraux avec donn\'ees dans $L\sp 1$. C. R. Acad. Sci. Paris
S\'er. I Math. \textbf{311} (1990), no. 10, 617--619.

\bibitem{BP} Br\'ezis, H., Ponce, A.: Reduced measures for obstacle
problems. Adv. Differential Equations  \textbf{10}  (2005),  no. 11,
1201--1234.

\bibitem{BS} Br\'ezis, H., Strauss, W.: Semi-linear second-order elliptic
equations in $L\sp{1}$. J. Math. Soc. Japan \textbf{25} (1973),
565--590.

\bibitem{CLR} {Chen, Y., Levine, S., Rao, M.:} Variable exponent, linear growth
functionals in image restoration. SIAM J. Appl. Math. {\bf 66}
(2006), no. 4, 1383--1406.

\bibitem{Cirmi} {Cirmi, G.R.:} Convergence of the solutions of nonlinear obstacle
problems with $L\sp 1$-data.  Asymptot. Anal.  \textbf{24}  (2000),
no. 3-4, 233--253.

\bibitem{DADM} Dall'Aglio, P., Dal Maso, G.: Some properties of the solutions of
obstacle problems with measure data. Ricerche Mat.  \textbf{48}
(1999), suppl., 99--116.

\bibitem{D} {Diening, L.}: Riesz potential and Sobolev embeddings on generalized
Lebes\-gue and Sobolev spaces $L\sp {p(\cdot)}$ and $W\sp
{k,p(\cdot)}$. Math. Nachr.  \textbf{268}  (2004), 31--43.

\bibitem{DHN} {Diening, L., H\"{a}st\"{o}, P., Nekvinda, A.}: Open problems in variable
exponent Lebesgue and Sobolev spaces. In: FSDONA04 Proceedings,
Drabek and Rakosnik (eds.), pp. 38--58, Milovy, Czech Republic,
2004.

\bibitem{ER00} {Edmunds, D., R\'akosn\'{\i}k, J.}: Sobolev embeddings with
variable exponent. Studia Math. {\bf 143} (2000), 267--293.


\bibitem{FZ} {Fan, X., Zhao, D.}:
A class of De Giorgi type and H\"{o}lder continuity. Nonlinear Anal.
{\bf 36} (1999), 295--318.

\bibitem{HHKV} {P. Harjulehto, P. H\"{a}st\"{o}, M. Koskenoja and S. Varonen}, \textit{The Dirichlet energy integral and variable exponent
Sobolev spaces with zero boundary values}, Potential Anal.
\textbf{25} (2006), 205-–222.

\bibitem{KinStam80} Kinderlehrer, D., Stampacchia, G.: An introduction to
variational inequalities and their applications. Pure and Applied
Mathematics \textbf{88}, Academic Press, New York--London, 1980.

\bibitem{KR} {Kov\'{a}\v{c}ik, O., R\'{a}kosn\'{\i}k, J.}:
On spaces $L^{p(x)}$ and $W^{1,p(x)}$. Czechoslovak Math. J.
\textbf{41} (1991), 592--618.

\bibitem{L1} Leone, C.: On a class of nonlinear obstacle problems with
measure data. Comm. Partial Differential Equations  \textbf{25}
(2000), no. 11-12, 2259--2286.

\bibitem{L2} Leone, C.: Stability results for
obstacle problems with measure data.  Ann. Inst. H. Poincar\'e Anal.
Non Lin\'eaire  \textbf{22}  (2005),  no. 6, 679--704.

\bibitem{L69} Lions, J.L.: Quelques m\'ethodes de r\'esolution des probl\`emes
aux limites non lin\'eaires. Dunod, Gauthier-Villars, Paris, 1969.

\bibitem{M} Mosco, U.: Implicit variational problems and quasi-variational
inequalities. Lecture Notes in Mathematics \textbf{543}, pp.
83--156, Spinger, 1976.

\bibitem{P} Palmeri, M.C.: Entropy subsolutions and supersolutions for
nonlinear elliptic equations in $L\sp 1$.  Ricerche Mat. {\bf 53}
(2004), 183--212.

\bibitem{Ro87} Rodrigues, J.F.: Obstacle problems in mathematical
physics. North-Holland Mathematics Studies \textbf{134},
North-Holland, Amsterdam, 1987.

\bibitem{Ro05} Rodrigues, J.F.: Stability remarks to the obstacle problem
for $p-$Lapla\-cian type equations. Calc. Var. Partial Differential
Equations {\bf 23} (2005), 51--65.

\bibitem{Ruzicka00} {R\r{u}\v{z}i\v{c}ka, M.}: Electrorheological fluids:
modeling and mathematical theory. Lecture Notes in Mathematics
\textbf{1748}, Springer-Verlag, Berlin, 2000.

\bibitem{SU06} Sanch\'on, M., Urbano, J.M.: Entropy solutions for the $p(x)-$Laplace
equation. Trans. Amer. Math. Soc., to appear.

\bibitem{Zh} {Zhikov, V.}: Averaging of functionals of the calculus of variations and
elasticity theory. Izv. Akad. Nauk SSSR Ser. Mat. \textbf{50}
(1986), no. 4, 675--710, 877.

\bibitem{Zh2} {Zhikov, V.}: Meyer-type estimates for solving the nonlinear Stokes
system. Differential Equations \textbf{33} (1997), no. 1, 108--115.

\end{thebibliography}
\end{document}